\documentclass[a4paper]{article}  

\usepackage{amssymb,amsmath,amsthm,verbatim}
\usepackage[dvips,final]{graphicx}
\usepackage{psfrag}
\usepackage{natbib}

\newcommand{\R}{\mathbb{R}}
\newcommand{\Prob}{\mathbb{P}}
\renewcommand{\Pr}{\Prob}

\newcommand{\E}{\mathbb{E}}

\newcommand{\td}{\tilde}

\newtheorem{theorem}{Theorem}[section]

\newtheorem{lemma}[theorem]{Lemma}

\theoremstyle{definition}

\newtheorem{remark}[theorem]{Remark}

\newcommand{\indic}[1]{\boldsymbol{1}_{\{\ensuremath{#1}\}}}

\newcommand{\as}{a.s.}

\newcommand{\eg}{e.g.}

\usepackage{hyperref}

\newcommand{\dd}{\mathrm{d}}
\newcommand{\spm}{m}

\title{{\sc Time-Homogeneous Diffusions with a Given Marginal at a
    Random Time}\thanks{Dedicated to Marc Yor on the occasion of his
    $60^{\textrm{th}}$ birthday.}}
\begin{document}
\author{Alexander Cox\thanks{e-mail:
        \texttt{A.M.G.Cox@bath.ac.uk};
        web: \texttt{www.maths.bath.ac.uk/$\sim$mapamgc/}}\\
        Dept.\ of Mathematical Sciences\\
        University of Bath\\
        Bath BA2 7AY
 \and David Hobson\thanks{e-mail:
        \texttt{D.Hobson@warwick.ac.uk};
        web: \texttt{www.warwick.ac.uk/go/dhobson/}}\\
        Department of Statistics\\
        University of Warwick\\
        Coventry CV4 7AL
 \and Jan Ob\l \'oj\thanks{e-mail:
        \texttt{obloj@maths.ox.ac.uk}; web:
        \texttt{www.maths.ox.ac.uk/$\sim$obloj/}}\\
        Mathematical Institute \emph{and}\\
	Oxford--Man Institute of Quantitative Finance\\
	University of Oxford\\
        Oxford OX1 3LB
      }
      
\maketitle
\begin{abstract}
  We solve explicitly the following problem: for a given probability
  measure $\mu$, we specify a generalised martingale diffusion $(X_t)$
  which, stopped at an independent exponential time $T$, is
  distributed according to $\mu$. The process $(X_t)$ is specified via
  its speed measure $m$. We present three proofs. First we show how
  the result can be derived from the solution of
  \cite{BertoinLeJan:92} to the Skorokhod embedding problem. Secondly,
  we give a proof exploiting applications of Krein's spectral theory
  of strings to the study of linear diffusions. Finally, we present a
  novel direct probabilistic proof based on a coupling argument.

\end{abstract}
\section{Introduction}

We are interested in the following general problem: suppose $\mu$ is a
given distribution on $\R$, and suppose $T$ is a (possibly random)
time.  Find a time-homogeneous martingale diffusion, $(X_t)$,
independent of $T$, for which $X_T$ is distributed according to
$\mu$. In particular, when $\mu$ is regular enough, we want to specify
a function $\sigma:\R \to \R_+$ such that
\begin{equation}\label{eq:martdef}
  X_T = \int_0^T \sigma(X_s) \, \dd W_s \sim \mu,
\end{equation}
where $(W_t)$ is a Brownian motion. The process $(X_t)$ is a diffusion
on natural scale described by its speed measure $m(\dd
x)=\sigma(x)^{-2}\dd x$.  When $\mu$ is less regular, the
interpretation of $m$ as a speed measure remains valid, but $m$ may no
longer have a density. In this case $X$ becomes a generalised or gap
diffusion.

In this paper we present a solution to the problem in the case where
$T$ is distributed exponentially with parameter 1. 
Somewhat surprisingly, not only does this
problem always have a solution but also the solution is fully
explicit. This can be seen both using a probabilistic and an analytic
approach. More precisely, in Section \ref{sec:BLJ} we exploit the
general solution to the Skorokhod embedding problem of
\cite{BertoinLeJan:92}. This construction is essentially based on the
theory of additive functionals for (nice) Markov processes.  Then, in
Section \ref{sec:KW}, we present a second proof based on the theory of
generalised diffusions as presented in \cite{KotaniWatanabe:82}. This
is a more analytic approach which hinges on duality between
generalised diffusions and strings and uses
Krein's spectral theory of strings.\\
Both of the above proofs exploit deep known results. In the final
section, we present a direct proof from first principles. First we
prove that the problem has a solution. We do this by writing $X$ as a
time-changed Brownian motion, $X_t=B_{A_t}$ and hence $A_T$ is a
solution to the Skorokhod embedding in $B$: that is, $B_{A_T}$ has
distribution $\mu$. Our proof relies on a specific coupling of
time-changes $A_t$ for different processes $X$. Furthermore, the
interpretation in terms of stopping times for $B$ gives an intuitive
justification for the explicit formula for $X$ (i.e.\ $\sigma$ in
\eqref{eq:martdef}).

\subsection{Motivation}
Our original goal was to solve the problem for the case where the time
$T$ is a fixed, positive constant. Then, a time-homogeneous diffusion
with marginal $T$ is in some sense the canonical, and perhaps
simplest, process consistent with a given (single) marginal. For the
problem with fixed $T$ and under sufficient smoothness conditions
\cite{LishangJiang} have proposed a relaxation scheme for calculating
the diffusion co-efficient. However, to the best of our knowledge,
existence and uniqueness for the problem with a general measure $\mu$
remain open.  One application of this result would be to mathematical
finance: the traded prices of call options with maturity $T$ can be
used to infer the marginal distribution of the underlying asset under
the pricing measure --- note also that the price process $(S_t)$,
suitably discounted, is a martingale under this measure --- so that
the solution of the problem for fixed $T$ would give a canonical candidate
price process consistent with market data. Other time-inhomogeneous
processes exist with the correct time $T$ marginals
(cf.~\cite{Dupire:94}, \cite{CarrMadan:98}, \cite{MadanYor:02}), and
the problem of finding examples, is related to the Skorokhod embedding
problem. For further discussion of the Skorokhod embedding problem,
and the connection to finance and model-independent hedging of
derivatives, see \cite{genealogia} or \cite{HobsonSurvey}.

Applications of the problem with $T$ exponentially distributed are
discussed in a recent paper of \cite{Carr:08}. Carr proposes modelling
the stock price process as a time-homogeneous diffusion time-changed
by an independent gamma subordinator: $S_t= X_{\gamma_t}$. The clock
is normalised so that $T=\gamma_{t^*}$ has an exponential
distribution, where $t^*$ is now the maturity of options whose prices
are known, so that \cite{Carr:08} effectively considers the same
problem as the present paper. His approach is to use forward-backward
Kolmogorov equations combined with Laplace transform properties. He is
able to derive explicitly $\sigma$ in \eqref{eq:martdef}, although he
only considers $\mu$ with positive density and does not prove general
existence or uniqueness results.

\section{Generalised (gap) diffusions}
We recall the classical construction of a generalised diffusion. Let
$m_i:[0,\infty]\to [0,\infty]$ be non-decreasing and right-continuous
with $m(\infty)=\infty$ and $\ell_i=\sup\{x:m_i(x)<\infty\}>0$, $i=1,2$.
Assume further that $m_2(0+)=0$. Then $\dd m_i$ are well
defined measures and we can define a measure $m$ on $\R$ by
\begin{equation}
  \label{eq:speedmeasure_def}
  m(\dd x)=\left\{
    \begin{array}{l}
      \dd m_1(x)\quad\mbox{ for } x\in [0,\infty), \\
      \dd \check{m}_2(x)\quad \mbox{ for } x\in (-\infty,0),
    \end{array}
  \right.
\end{equation}
where $\dd \check{m}_2$ is the image of $\dd m_2$ under $x\to -x$. Naturally, $\ell_{i}$ can be defined directly from $m$ and we write $\ell_-=\ell_-(m)=-\ell_2$ and $\ell_+=\ell_+(m)=\ell_1$.\\
Consider $(B_t,\Pr^{x_0})$ a one-dimensional Brownian motion defined on
$(\Omega,\mathcal{F},(\mathcal{F}_t))$, with $B_0=x_0$, $\Pr^{x_0}$-a.s. 
We
assume $\mathcal{F}_0$ is rich enough to support random variables
independent of $B$. Let $(L^x_t)$ be the jointly continuous version of
the local time of $(B_t)$. We adopt here the classical It\^o--McKean
normalisation in which $|B_t-x|-L^x_t$ is a martingale. Put
$\Phi_t=\int_\R L^x_t m(\dd x)$ and let $(A_t)$ be the
right-continuous inverse of $(\Phi_t)$. Then
\begin{equation}
  \label{eq:gen_diff_def}
  (X_t,\Pr^{x_0}),\quad X_t:=B_{A_t},\quad \textrm{for } t\leq \zeta=
  \inf\{t\geq 0: X_t \notin (\ell_-,\ell_+)\}
\end{equation}
is a time-change of a Brownian motion and hence a strong Markov
process living on $supp(m)$.  It is called a \emph{generalised
  diffusion} (on natural scale) corresponding to the measure $m$. It
has also been called a \emph{gap diffusion} in \cite{Knight:81}. Note
that, due to our normalisation, the local time $L^x_t$ is twice the
local time in \cite{KotaniWatanabe:82}. In consequence $(X_t)$ is a
generalised diffusion corresponding to measure $2m$ in the notation of
\cite{KotaniWatanabe:82}. As an example, in this paper Brownian motion
is a diffusion with speed measure equal to the Lebesgue measure and
not twice the Lebesgue measure as in \cite{KotaniWatanabe:82}.

In order to understand better the relationship between features of $m$
and the behaviour of $X$ we discuss two important classes, firstly
where $m$ has a positive density, and secondly where $m$ only charges
points.

Suppose first that $0<m([a,b])<\infty$ for any
$\ell_-<a<b<\ell_+$. Then $(X,\Pr^x)$ is a
regular\footnote{i.e.~ $\Pr^x(H_y < \infty) >0$ for all $x,y \in I$,
  where $H_y = \inf\{t > 0 : X_t = y\}$.}  diffusion on
$I=[\ell_-,\ell_+]$ with absorbing boundary points. $X$ is
on natural scale and $m(\dd x)$ is simply its speed measure. We have
$A_t=[X]_t$ and the measure $m$ can be recovered from $X$ as
\begin{equation*}
  m(\dd x)=-\frac{1}{2}h_J''(\dd x),\quad h_J:=\E^x[\inf\{t:X_t\notin
  J\}],\ J=[a,b]\subset I,
\end{equation*}
since it can be shown that $h_J(x)$ is convex and, for $J\subset K$,
as measures $h_J''=h_K''$ on $int(J)$, see
\cite[Sec.~V.47]{RogersWilliams:00b} for a detailed discussion.  If
further $m(\dd x)=\lambda(x)\dd x$, with $\lambda$ bounded and
uniformly positive on $I$, then $(X_t)$ solves
\begin{equation}\label{eq:SDE_X}
  \dd X_t=\lambda(X_t)^{-1/2}\dd W_t,\quad t<\zeta,
\end{equation}
for a Brownian motion $(W_t)$. Equivalently the infinitesimal
generator of $X$, when acting on functions supported on $I$, is
$\mathcal{G}=\frac{1}{2\lambda(x)}\frac{\dd^2}{\dd
  x^2}=\frac{1}{2}\frac{\dd^2}{\dd m\dd x}$. Note that then
$A_t=[X]_t=\int_0^t \lambda(X_s)^{-1}\dd s$ and it can be verified
directly from the occupation time formula that $\Phi_t^{-1}=[X]_t$:
\begin{equation}\label{eq:phi_inv_A}
  \begin{split}
    \Phi_{[X]_t}&=\int_{\R}L^x_{[X]_t}m(\dd x)=\int_{\R}L^x_{[X]_t}\lambda(x)\dd x\\
    &=\int_0^{[X]_t} \lambda(B_s)\dd s=\int_0^t\lambda(B_{[X]_u})\dd
    [X]_u=\int_0^t\lambda(X_u)\dd [X]_u=t.
  \end{split}
\end{equation}
From the above discussion we see that the regions where $m$ has more
mass correspond to the regions where $(X_t)$ moves more slowly (and thus
spends more time).

By a natural extension of the above analysis, if $m$ has zero mass in
an interval, then since $(X_t)$ lives on the support of $m$, this
interval is not visited by $(X_t)$ at all. Conversely, if
$m(\{a\})>0$ (and for every neighbourhood $U$ of $a$, $m$ charges $U\setminus \{a\}$), then $a$
is a sticky point for $X$: started in $a$, $(X_u:u\leq t)$ spends a
positive time in $a$ even though it exits $a$ instantaneously.

It remains now to understand the role of isolated atoms in $m$.
Consider $m=\sum_{i=1}^N \beta_i\delta_{a_i}$ for $a_1<\ldots<a_N$, 
$\beta_i>0$ and $\beta_1=\beta_N=\infty$. Then $X$ is a continuous-time 
Markov chain living on
$\{a_1,\ldots,a_N\}$, stopped upon hitting $a_1$ or $a_N$. From the
construction it follows that $X$ can only jump to nearest neighbours,
i.e.\ from $a_i$ it can jump to $a_{i-1},a_{i+1}$, and the
probabilities of these follow instantly from the fact that $(X_t)$ is
a martingale. The time spent in $a_i$, $1<i<N$, before $(X_t)$ jumps to 
a next
point, has an exponential distribution with mean
$$\beta_i\E^{a_i}[L^{a_i}_{H_{a_{i-1},a_{i+1}}}]=2\beta_i\frac{(a_{i+1}-a_i)(a_i-a_{i-1})}{a_{i+1}-a_{i-1}},$$
where $H_{a,b}$ is the first hitting time of $\{a,b\}$ for $(B_t)$.
(This is an example of the more general formula (for $a < x\wedge y\leq x\vee y < b$)
\begin{equation}\label{eq:Elocaltime} \E^x[ L^y_{H_a \wedge H_b} ] = 2
  \frac{ ( x \wedge y - a )(b-x 
    \vee y)}{b-a} \end{equation}
for expected values of Brownian local times.) This completes the description of $X$. We see that an isolated atom in $a_i$ has an effect of introducing a holding time in that point for $X$, with mean proportional to $m(\{a_i\})$.

\section{Main result and probabilistic proof}\label{sec:BLJ}
Having recalled generalised diffusions we can now state the main
result of our paper. In this section we provide a proof rooted in
probabilistic arguments while in the next section we describe a more
analytical approach. For a probability measure $\mu$ we let $\ell_-^\mu$ and $\ell_+^\mu$ denote respectively the lower and the upper bounds of its support.
\begin{theorem}\label{thm:main}
  Let $\mu$ be a probability measure, $\int |x|\mu(\dd x)<\infty$, $\int x\mu(\dd x)=x_0$ and let $u_\mu(x)=\int_{\R}
  |x-y|\mu(\dd y)$. Define a measure $m$ by 
  \begin{equation}\label{eq:m_measure_def}
\begin{split}
    &m(\dd x) = \frac{\mu(\dd x)}{u_\mu(x)-|x-x_0|} \qquad\qquad\ \textrm{ for }x\in (\ell_-^\mu,\ell_+^\mu),\\
    &m([y,x_0))=m([x_0,x])=\infty\qquad \qquad  \textrm{ for } y\leq \ell_-^\mu\leq \ell_+^\mu\leq x\ .
\end{split}
\end{equation}
  Let $(X_t)$ be the generalised diffusion associated with $m$ and $T$
  be an $\mathcal{F}_0$--measurable $\Pr^{x_0}$--exponential random
  variable independent of $(X_t)$. Then, under $\Pr^{x_0}$,
  $X_T\sim\mu$ and $(X_{t\land T})$ is a uniformly integrable
  martingale.
\end{theorem}
\begin{proof}
  It suffices to prove the theorem for $x_0=0$ as the general case
  follows by a simple shift in space.  Assume in first instance that
  $supp(\mu)\subset (-N,N)$. Consider the following process $(Y_t)$:
  it takes values in $(-N,N)\cup \{\star\}$, with $\star$ added as an
  isolated point. $Y$ starts in $\star$ which is a holding point with
  parameter $1$. On exit from $\{\star\}$ at time $\rho_\star$, the
  process behaves as $B$ under the measure $\Pr^{0}$, so that
  $Y_{\rho_\star+t}=B_t$, until exit from $(-N,N)$ when $Y$ jumps back
  to $\star$. In this way $(Y_t)$ is a recurrent strong Markov process
  with $\star$ as its regular starting point. Write $\tilde{\Pr}^x$
  for the probability measure associated with the process $Y_t$
  started at $x$, noting that for all $\tilde{\Pr}^x$, the path jumps
  from $\star$ to $0$. We make explicit the \cite{BertoinLeJan:92}
  solution to the Skorokhod embedding problem
  of $\mu$ in $Y$.\\
  Let $\tau_\star=\inf\{t>\rho_\star: Y_t=\star\}
  =\rho_\star+\inf\{t>0: B_t\notin (-N,N)\}=:\rho_\star+H$. The
  process $(Y_t)$ admits a family of local times $(L^a_t(Y))$. We
  simply have $L^a_t(Y)=L^a_{t-\rho_\star}$, $|a|<N$ and
  $L^\star_t(Y)=L^\star_{\rho_\star}(Y)$ for $\rho_\star\leq
  t<\tau_\star$. This last quantity is exponentially distributed and
  independent of $(B_t)$. It follows from \eqref{eq:Elocaltime} that
  \begin{equation}
    \begin{split}
      \tilde{\E}^{\star}[L^a_{\tau_\star}(Y)] & =
      \E^{0}[L^a_H]=N-|a|,\quad
      |a|<N,\\
      \tilde{\E}^x[L^a_{\tau_\star}(Y)] & = \E^{x}[L^a_H] =\frac{(a
        \wedge x+N)(N - a \vee x)}{N}, 
      \ |a|<N,\ |x|<N\ .
    \end{split}
  \end{equation}
  The invariant measure $\nu$ for $Y$, displayed in $(1)$ in
  \cite{BertoinLeJan:92}, acts by
  \begin{equation}
    \begin{split}
      \int f \dd \nu &=
      \tilde{\E}^{\star}\left[\int_{\rho_\star}^{\tau_\star}f(Y_s)\dd
        s\right]+f(\star)=\E^{0}\left[\int_0^H f(B_u)\dd
        u\right]+f(\star)\\
      &=\int_{-N}^{N} f(a)\E^{0}[L^a_H]\dd a+f(\star)=\int_{-N}^{N}
      f(a)\left(N-|a|\right)\dd a+f(\star)\ .
    \end{split}
  \end{equation}
  Consider a finite positive measure $\xi$ on $(-N,N)$ and a positive
  continuous additive functional $F_t=\int L^a_t(Y)\xi(\dd a)$. The
  Revuz measure $\chi$ of $F$ is then given by
  \[
  \int f \dd \chi = \frac{1}{t}\tilde{\E}^{\nu}\left[\int_0^t
    f(Y_s)\dd F_s\right] = \frac{1}{t}\int f(a)
  \tilde{\E}^{\nu}[L^a_t(Y)]\xi (\dd a)=\int f(a)(N-|a|)\xi(\dd a)
  \]
  so that $\xi(da)=\chi(da)/(N- |a|)$ and $\chi=\mu$ iff $\xi(\dd
  a)=\frac{\mu(\dd a)}{N-|a|}$. We proceed to compute $V_\chi$ and
  $\hat{V}_\mu$, as defined in \cite{BertoinLeJan:92}. We have 
$V_\chi(x)=\tilde{\E}^x [ \int_0^{\tau_\star}\dd F_s ]$ and 
$\int \hat{V}_\mu \dd \chi = \int V_\chi \dd \mu$. Then,
for $x \in (-N,N)$, 
  \begin{equation}
    \begin{split}
      V_\chi(x)&=
      \int
      \E^x[L^a_H]\frac{\chi(\dd a)}{N-|a|}\\
      &=\int\frac{(N+ a \wedge x)(N - a \vee x)}{N ( N-|a|) }\chi(\dd
      a),
    \end{split}
  \end{equation}
and it follows that
  \begin{equation} \label{eq:Vhatmu}
    \hat{V}_\mu(a)=\int \frac{(N+ a 
      \wedge x)(N - a \vee x)}{N ( N-|a|) } \mu(\dd 
    x)=\frac{N-u_\mu(a)}{N-|a|}\ .
  \end{equation}
  Jensen's inequality grants us $u_\mu(a)\geq |a|$ and hence
  $\hat{V}_\mu\leq 1$ is bounded as required.
  Furthermore, from \eqref{eq:Vhatmu}, for $0<a<N$
  \[
  \hat{V}_{\mu}(a) \geq \frac{\int_{a}^N u_{\mu}'(x) dx}{\int_{a}^N
    dx} \geq u_{\mu}'(a+) = 1 - 2 \mu((a,\infty)) \stackrel{a \uparrow
    N}{\rightarrow} 1
  \]
  since $\mu$ has support in $(-N,N)$. Hence
  the bound $\hat{V}_{\mu}(a) \leq 1$ is best possible.

  We have $$(1-\hat{V}_\mu(a))^{-1}=\frac{N-|a|}{u_\mu(a)-|a|}$$ and
  the \cite{BertoinLeJan:92} stopping time is given by
  \begin{equation} \label{eq:BLJ_st1}
    \begin{split}
      T_{BLJ}& =\inf\left\{t\geq 0: \int_0^t
        \frac{N-|Y_s|}{u_\mu(Y_s)-|Y_s|}\dd F_s >
        L^\star_t(Y)\right\}\ ,
    \end{split}
  \end{equation}
  where $F_t=\int L^a_t(Y)\frac{\mu(\dd a)}{N-|a|}$. The key result of
  \cite{BertoinLeJan:92} is that $T_{BLJ}$ solves the Skorokhod
  embedding problem for $\mu$: i.e. $Y_{T_{BLJ}}\sim \mu$.  Recall
  that $X_t=B_{A_t}$ with $A_t$ the right--continuous inverse of
  $\Phi_t=\int L^a_t m(\dd a)$, where $m$ is as displayed in
  \eqref{eq:m_measure_def}.  By the Corollary in
  \cite[p.~540]{BertoinLeJan:92}, $\E^{0}[L^\star_{T_{BLJ}}(Y)]=1$ from
  which it follows that $T_{BLJ}< \tau_\star$. This allows us to
  rewrite $T_{BLJ}$ as
  \begin{equation}
    \begin{split}
      T_{BLJ}& = \rho_\star+ \inf\left\{t\geq 0: \int_0^t \frac{N-|B_s|}{u_\mu(B_s)-|B_s|}\dd F_{s-\rho_\star} > L^\star_{\rho_\star}(Y)\right\}\\
      & = \rho_\star+ \inf\left\{t\geq 0: \int \frac{L^a_t\mu(\dd a)}{u_\mu(a)-|a|} > L^\star_{\rho_\star}(Y)\right\}\\
      & = \rho_\star+ \inf\left\{t\geq 0: \Phi_t >
        L^\star_{\rho_\star}(Y)\right\}=\rho_\star+A_{L^\star_{\rho_\star}(Y)}\
      .
    \end{split}
  \end{equation}
  Hence $B_{A_T}=X_T\sim \mu$, with $T={L^\star_{\rho_\star}(Y)}$ an
  exponential random variable, as required. Note that by construction
  $(X_t)$ remains within the bounds of the support of $\mu$. In
  particular, uniform integrability of $(X_{t\land T})$ follows from
  the fact that it is a bounded martingale.\\
  Now consider an arbitrary $\mu$ and $m$ defined via
  \eqref{eq:m_measure_def}. Note that $\ell_-^\mu=\ell_-(m)$ and
  $\ell_+^\mu=\ell_+(m)$. Naturally if $\mu$ has a bounded support
  then the previous reasoning applies, so suppose that $-\ell_-^\mu=\ell_+^\mu=\infty$. For $M > |u_\mu(0)|$, let $\mu_M$ be the measure on
  $[q^-_M,q^+_M]\cup\{-M,M\}$, centred in zero and with
  $u_{\mu_M}=u_{\mu}$ on $[q^-_M,q^+_M]$, and $u_{\mu_M}\le
  u_{\mu}$. Note that this defines $q_M^\pm$ and $\mu_M$ uniquely,
  $q_M^\pm$ converge to the bounds of the support of $\mu$ as
  $M\to\infty$, $\mu_M=\mu$ on $(q^-_M,q^+_M)$, 
$\mu(\{q^\pm_M\}) \ge \mu_M(\{q^\pm_M\})$ 
and $\mu_M$ converges weakly 
to $\mu$, see
  \cite{Chacon:77} for details. Let $A^M_t$ be the inverse of
  $\Phi^M_t=\int L^x_t m_M(\dd x)$, with $m_M$ given by
  \eqref{eq:m_measure_def} for $\mu_M$, and $X^M_t=B_{A^M_t}$. Fix an
  exponential random variable $T$ independent of $(B_t)$. By the
  construction above (applied with $N=M+1$, so that the support of
  $\mu_M$ is contained in $(-N,N)$), $X^M_T\sim \mu_M$. Observe that,
  since $m_M=m$ on $(q^-_M,q^+_M)$, we have $X^M_t=X_t$ for $t <
  \tau_M:=\inf\{t: X_t\notin (q^-_M,q^+_M)\}$. Since $\Pr(T \ge
  \tau_M) \to 0$, both $X^M_{T}$ and $X_{T}\indic{T < \tau_M}$
  converge to the same limit in distribution as $M \to \infty$, and
  hence $X_T \sim \mu$.\\
  To see that the process $X_{t \land T}$ is uniformly integrable, we
  note that this is equivalent to the process $B_{A_{t \land T}}$
  being uniformly integrable. This is easy to see as a consequence of
  Proposition~18 of \cite{Cox:08}, noting that since our laws $\mu_n,
  \mu$ are centered, uniform integrability and minimality are
  equivalent.\\
  Finally, when only one of $\ell_-^\mu,\ell_+^\mu$ is infinite, say $\ell_+^\mu=\infty$, the procedure is analogous but we only truncate the support of $\mu$ on one side, i.e.\ we look at $\mu_M=\mu$ on $(q^-_M,\infty)$.

\end{proof}

\section{An analytic proof}\label{sec:KW}
In the previous section, we proved Theorem \ref{thm:main} using a
probabilistic approach, characteristic of the work of
\cite{BertoinLeJan:92}).  However, study of generalised diffusions can
be seen as a probabilistic counterpart of the theory of strings, see
\cite{DymMcKean:76}. Indeed the theory of strings and original results
in \cite{Krein:52} have played an important r\^ole in the study of
fine properties of generalised diffusions including the L\'evy
measures of their inverse local times\footnote{Essentially Krein's
  theorem provides a bijection between the set of strings and the set
  of their spectral measures. This is equivalent with a bijection
  between generalised diffusions $(X_t)$ with $m_2\equiv 0$, reflected
  in zero, with the set of subordinators given by the inverse of the
  local time in zero of $X$.}, lower bounds on the spectrum of their
infinitesimal generator, asymptotics of their transition densities and
first hitting times distributions see
\cite{KacKrein:58,Knight:81,KotaniWatanabe:82,KuchlerSalminen:89}. It
would be thus natural to re-derive \eqref{eq:m_measure_def} using
analytic methods. This is indeed possible and we sketch here the main
steps of such derivation.

Let $(X_t)$ be a generalised diffusion associated with a measure $m$
as in \eqref{eq:speedmeasure_def}. Recall that in the notation of
\cite{KotaniWatanabe:82} $X$ is associated to measure $2m$. Let 
$\phi,\psi,h_{\pm}(\lambda),h(\lambda)$ and $u_{\pm}$ be defined as in
(3.1)--(3.4) in \cite{KotaniWatanabe:82} (but with our normalisation of 
$m$):
\begin{align*}
  \phi(x,\lambda) & =
  \begin{cases}
    1 + 2 \lambda \int_{0-}^{x+}(x-y)
    \phi(y,\lambda) m(\dd y) &: 0\leq x<\ell_+ \\
    1 + 2 \lambda \int_{x-}^{0-}(y-x) \phi(y,\lambda) m(\dd y) &: \ell_-<x < 0
  \end{cases}\\
  \psi(x,\lambda) & =
  \begin{cases}
    x + 2 \lambda \int_{0-}^{x+}(x-y) \psi(y,\lambda) m(\dd y) &: 0\leq x<\ell_+ \\
    x + 2 \lambda \int_{x-}^{0-}(y-x) \psi(y,\lambda) m(\dd y) &: \ell_-<x < 0
  \end{cases}\\
  h_+(\lambda) & = \int_0^{\ell_+} \frac{1}{\phi(x,\lambda)^2} \, \dd
  x = \lim_{x \uparrow \ell_+} \frac{\psi(x,\lambda)}{\phi(x,\lambda)}
  \\
  h_-(\lambda) & = \int_{\ell_-}^0 \frac{1}{\phi(x,\lambda)^2} \, \dd
  x = -\lim_{x \downarrow \ell_-}
  \frac{\psi(x,\lambda)}{\phi(x,\lambda)}
  \\
  \frac{1}{h(\lambda)} & = \frac{1}{h_+(\lambda)} +
  \frac{1}{h_-(\lambda)} \\
  u_{\pm}(x,\lambda) & = \phi(x,\lambda) \mp
  \frac{\psi(x,\lambda)}{h_{\pm}(\lambda)} .
\end{align*} 
These functions yield a direct representation (3.5) therein:
\begin{equation}\label{eq:geq}
  g_\lambda(x,y)=g_{\lambda}(y,x)=h(\lambda)u_+(x,\lambda)u_-(y,\lambda),\quad x\geq y,
\end{equation}
of the resolvent density $g_\lambda$, defined by
\begin{equation}
  \E^x\left[\int_0^\zeta \mathrm{e}^{-\lambda t} f(X_t)\dd t\right]=\int_{(\ell_-,\ell_+)}
  2g_\lambda(x,y)f(y)m(\dd y), \quad x\in (-\ell_-,\ell_+)\label{eq:gdef}
\end{equation}
for continuous bounded functions $f$ on the support of $m$.\\
In what follows we take $\lambda=1$ and drop the $\lambda$
argument. We have $u_\pm(0)=1$ and it can be checked independently 
(or deduced from \eqref{eq:geq}--\eqref{eq:gdef} above) that 
$u_\pm$ are non-negative with $u_+$ non-increasing, and $u_-$ 
non-decreasing. Further we have $u_+(x)\to 0$ as $x\to 
\ell_+$ and $u_-(x)\to 0$ as $x\to \ell_-$. This is described in detail, in the case of standard diffusion processes, in Theorem 5.13.3 in \cite{Ito:06} (note that our $u_+$ is the solution $\underline{u}$ therein), see also \cite[p.~241]{KotaniWatanabe:82} and \cite[p.~57]{Knight:81}. Furthermore, from their definitions, we have
\begin{equation}\label{eq:for_upm}
\begin{split}
  u_+''(\dd x)&= 2u_+(x)m(\dd x),\qquad 0<x<\ell_+,\\ u_-''(\dd x)&= 2u_-(x)m(\dd x),\qquad \ell_-<x<0\ .
\end{split}
\end{equation}
Assume for simplicity that $\mu$ is a centred probability measure:
$x_0=0$, and put $U_\mu(x)=u_\mu(x)-|x|$.  Then, Theorem
\ref{thm:main} is simply equivalent to showing that for $\lambda=1$
and $m$ given by \eqref{eq:m_measure_def}, that is $m(\dd
x)=\frac{\mu(\dd x)}{U_\mu(x)}$, we have $g_1(x,0)=\frac{1}{2}U_\mu(x)$.\\
Observe now that $U_\mu(x)$ also solves \eqref{eq:for_upm} and $U_\mu(x)\to 0$ as $|x|\to\infty$. Given that
$g_\lambda$ solves \eqref{eq:geq} it follows that
$g(x)=g_1(0,x)=cU_\mu(x)$ for some constant $c$. It remains to show
that $c=\frac{1}{2}$. For this we analyse the derivative in zero. On
the one hand we have
$$U_\mu'(0-)-U'_\mu(0+)=2\mu([0,\infty))+2 \mu((-\infty,0))=2$$
and on the other hand
$$g'(0-)-g'(0+)= 
h(\lambda)\left(u_-'(0-)-u_+'(0+)\right)=h(\lambda)\left(\frac{1}{h_-(\lambda)}+\frac{1}{h_+(\lambda)}\right)=1,$$
from which it follows that $c=\frac{1}{2}$ as required.

To end this section we check that the boundary behaviour of $X$ at $\{\ell_-,\ell_+\}$ is what we would expect it to be. Consider for example $\ell_+$. 
If $\mu(\{\ell_+\})=0$ then we have 
\begin{equation}
\begin{split}
m((0,\ell_+))&=\int_{(0,\ell_+)}\frac{\mu(\dd x)}{\int_x^{\ell_+} (y-x)\mu(\dd y)}\geq \int_{\ell_+-\epsilon}^{\ell_+}\frac{\mu(\dd x)}{\int_x^{\ell_+} (y-x)\mu(\dd y)}\\
&\geq \frac{1}{\epsilon} \int_{\ell_+-\epsilon}^{\ell_+}\frac{\mu(\dd x)}{\mu([x,\ell_+])}=\frac{1}{\epsilon}\left(\log \mu([\ell_+-\epsilon,\ell_+])-\log \mu(\{\ell_+\})\right)=\infty.
\end{split}
\end{equation}
In a similar manner we have
\begin{equation}
\begin{split}
\sigma_m:&=\int_{(0,\ell_+)}m((0,x])\dd x = \int_0^{\ell_+}\int_0^x \frac{\mu(\dd y)\dd x}{\int_y^{\ell_+} (u-y)\mu(\dd u)}=\int_0^{\ell_+}\frac{(\ell_+-y)\mu(\dd y)}{\int_y^{\ell_+} (u-y)\mu(\dd u)}\\
&\geq \int_0^{\ell_+}\frac{(\ell_+-y)\mu(\dd y)}{\int_y^{\ell_+} (\ell_+-y)\mu(\dd u)}=\int_0^{\ell_+}\frac{\mu(\dd y)}{\mu([y,\ell_+])}=\infty.
\end{split}
\end{equation}
Hence, by definition, $\ell_+$ is a natural boundary for $(X_t)$ which can not be reached in finite time starting from a point $x<\ell_+$, see \cite[Sections 5.11 and 5.16]{Ito:06}.

Suppose $\mu(\{\ell_+\})>0$. First note that $\ell_+$ is a trap since $m(\{\ell_+\})=\infty$ by the definition of $m$ in \eqref{eq:m_measure_def}, and hence $(X_t)$ is absorbed in $\ell_+$ upon reaching it. Further we have
\begin{equation}
\begin{split}
\sigma_m&=\int_{(0,\ell_+)}\frac{(\ell_+-y)\mu(\dd y)}{\int_y^{\ell_+} (u-y)\mu(\dd u)}\leq \int_{(0,\ell_+)}\frac{(\ell_+-y)\mu(\dd y)}{\int_{\frac{y+\ell_+}{2}}^{\ell_+} (u-y)\mu(\dd u)}\\
& \leq 2\int_{(0,\ell_+)}\frac{\mu(\dd y)}{\mu([\frac{y+\ell_+}{2},\ell_+])}<\infty.
\end{split}
\end{equation}
It follows that if $\ell_+$ is the endpoint of a regular interval for $(X_t)$, i.e.\ if $m((\ell_+-\epsilon,\ell_+))>0$ for all $\epsilon>0$ then $\ell_+$ is a regular or exit boundary and hence the process $(X_t)$ can reach $\ell_+$ in finite time. 
Finally, if $\ell_+$ is an isolated point in the support of $\mu$, i.e.\ if $\mu((\ell_+-\epsilon,\ell_+])=\mu(\{\ell_+\})$ for some $\epsilon>0$, then likewise it is an isolated absorbing point in the state space of $(X_t)$. It is easy to see that it can be reached in finite time with positive probability by considering $(X_t)$ reaching the point $\sup\{x\in supp(\mu): x<\ell_+\}$ and behaving thereafter.

\section{A more intuitive and direct proof}\label{sec:our}
We have so far presented two methods of arriving at the representation
\eqref{eq:m_measure_def} and Theorem~\ref{thm:main}. Both relied on
deep probabilistic or analytic results and neither method appears to
give a strong insight into why the result might be true. Consequently,
one might want to have a more bare-handed proof, particularly if one
wishes to generalise the result to other settings. Our goal in this
section, is to give a direct proof of Theorem \ref{thm:main}, using a
coupling and a construction of a martingale diffusion as a time-change
of Brownian motion. 
The intuitive picture on which we base our proofs exploits the fact that
we can write a time-changed martingale
diffusion as a Brownian motion. In this picture, `locally', the
process would appear to stop according to an exponential random
variable, whose parameter would depend on the speed of the diffusion
at that location; generalising this idea, we propose modelling the
choice of an exponential stopping time by a Poisson Random measure on
$\R_+ \times \R$, where points are placed with intensity $\dd u \,
m(\dd x)$, i.e.\ with more frequency in places where we expect to stop
more often. 
Then we stop the process at a point $x$, if there is a point at
$(u,x)$ in the Poisson random measure, and if the local time of the
Brownian motion at $x$ reaches $u$, before the local time at any other
$x'$ reaches $u'>0$ for some other point $(u',x')$ of the Poisson
random measure.  By comparing these stopping times $T^m$ derived from
different Poisson random measures, we are able to prove a monotonicity
result. This gives us a coupling argument from which we deduce the
existence of a measure $m$ with the desired stopping distribution,
i.e.\ $B_{T^m}\sim \mu$. Some simple calculations show that
construction of a suitable generalised diffusion follows.  This new
insight then allows us to give an intuitive justification of the
explicit formula \eqref{eq:m_measure_def}.  Observe that effectively
we re-interpret the original problem as a new problem of finding a
solution to the Skorokhod embedding problem (cf.~\cite{genealogia}) in
a given class of stopping times $T^m$.

We fix the underlying Brownian motion $(B_t)$ and the stopping times
will be based on its local times $L^x_t$.  We think about the
behaviour of the process in the context of the curve of the local
times $L^x_t$ of $B_t$ as $t$ is increasing. More specifically, define
\[
R_t = \{(u,x) : L_t^x > u\}
\]
which is the set of points `inside' the local time curve. Now, given a
measure $\spm(\dd x)$, we suppose $\Delta^\spm$ is a Poisson Random
Measure on $\R_+ \times \R$ with intensity measure $\dd u\,\spm(\dd
x)$, independent of the Brownian motion $B$.  We allow $\spm(\dd x)$
to be infinite on some intervals. More precisely we assume that there
exists a, possibly infinite, interval $I$, containing the origin, such
that $\spm(\Gamma)<\infty$ for any compact set $\Gamma\subset I$ and
that $\spm_{|I^c}\equiv \infty$. This agrees with
\eqref{eq:speedmeasure_def}.  Formally the measure $\Delta^\spm$
decomposes into $\Delta^\spm=\td{\Delta}+\Delta_\infty$ where
$\td{\Delta}$ is a Poisson Random measure with intensity $\dd u\,
\spm_{|I}(\dd x)$ and $\Delta_\infty=\sum_{x\notin
  I}\delta_{(0,x)}$. We adopt this convention from now on.

We define the stopping time
\begin{equation} \label{eqn:Tdefn} T^{\spm} = \inf\{t \ge 0 :
  \Delta^\spm(R_t) \geq 1\}.
\end{equation}
Figure~\ref{fig:basic} shows a graphical representation of the
construction.  \psfrag{L}{$L_t^x$} \psfrag{x}{$x$} \psfrag{F1}{$F_1$}
\psfrag{F2}{$F_2$}
\begin{figure}[t]
  \begin{center}
    \includegraphics[width=3in,height=3in]{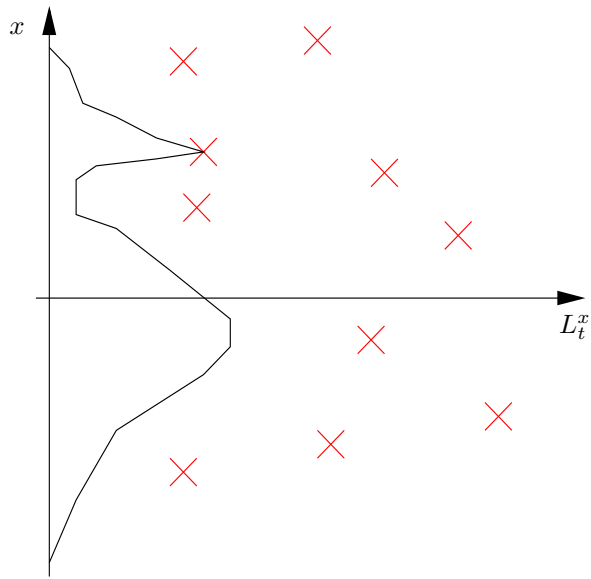}
    \caption{\label{fig:basic} The curve on the left represents the
      local time $L_t^x$ at time $t$. As time increases, the curve
      moves outwards. The crosses are distributed according to
      $\Delta$ and $T$ occurs the first time the local time curve
      hits such a point.} \end{center}
\end{figure}
The idea now is to construct $\spm$ such that $B_{T^\spm}\sim \mu$ and
then time change with $A^m_t$ to obtain the desired generalised diffusion process as in \eqref{eq:gen_diff_def}. This is
explained in the following theorem.

\begin{theorem}\label{thm:main2}
  Given a centered probability distribution $\mu$ on $\R$ there exists
  a measure $\spm$ such that $T^\spm$ is minimal\footnote{Following
    \citet{monroe_embedding_1972}, a stopping time $T$ is minimal if
    $S \le T$ \as{} and $B_S \sim B_T$ imply $S=T$ \as. Then if $B_T$
    is a centred random variable, and $B_0 = 0$, minimality of $T$ is
    equivalent to the uniform integrability of $B_{t \wedge T}$. } and
  embeds $\mu$:
  $B_{T^\spm}\sim \mu$.\\
  Furthermore, if $(X^m_t)$ is a generalised diffusion given via
  \eqref{eq:gen_diff_def} then the stopping time
  \begin{equation}\label{eq:Tsimdef}
    \td{T}^\spm = A^{-1}_{T^\spm}
  \end{equation}
  is exponentially distributed, independently of $X$, and
  $X_{\td{T}^\spm}=B_{T^\spm}\sim \mu$.
\end{theorem}
\begin{remark} {\rm The above statement has two advantages in our
    opinion. Firstly, it provides an additional insight into the
    relation between stopping times for $X$ and $B$. Secondly, it can
    be proved using a fairly direct and elementary arguments.  We note
    however that it is poorer than Theorem \ref{thm:main} which not
    only gives existence of $\spm$ but also the explicit formula
    \eqref{eq:m_measure_def}. }
\end{remark}
\begin{proof}
We prove the first part of the theorem in two steps: in Step 1 we assume
 that $\mu$ has bounded support and in Step 2 we extend the result to
 arbitrary $\mu$ via a limiting proceedure. Finally we prove the second
  part of the theorem.\\
  Whenever no ambiguity is possible we suppress the superscripts
  $\spm$.\\
  \emph{Part I: Step 1}. We assume that $\mu$ has bounded support and
  denote by $[\ell_-,\ell_+]$ the smallest closed interval with
  $\mu([\ell_-,\ell_+])=1$. Define the set
  \begin{equation}\label{eq:setS}
    S^\mu=\left\{\spm: \forall\, \Gamma\subset (\ell_-,\ell_+),\,
      \Pr(B_{T^\spm}\in \Gamma)\le \mu (\Gamma)\right\}.
  \end{equation}
  We will now show that this set has non-trivial elements.  Fix
  $\epsilon>0$ and let $\spm_\epsilon=\epsilon \mu$ on $(\ell_-,\ell_+)$ and
  infinity elsewhere. We have then $T^{\spm_\epsilon}\le H_{\ell_-}\land
  H_{\ell_+}$, where $H_a=\inf\{t: B_t=a\}$, and thus $\E
  L^x_{T^{\spm_\epsilon}}\le 2\frac{\ell_+|\ell_-|}{\ell_+-\ell_-}$. Let $\Gamma\subset
  (\ell_-,\ell_+)$ and define:
  \begin{equation*}
    R_t^{\Gamma} = \{(u,x): L_t^x > u, x \in \Gamma\}.
  \end{equation*}
  Then we deduce:
  \begin{eqnarray}
    \Pr(B_{T^{\spm_\epsilon}}\in \Gamma) & \le &
    \Pr\Big(\Delta^{m_\epsilon}(R^\Gamma_{H_{\ell_-}\land
  H_{\ell_+}})\ge 1\Big) \\ 
    &\le& \E\big[ \Delta^{m_\epsilon}(R^\Gamma_{H_{\ell_-}\land
  H_{\ell_+}})\big] =
    \epsilon\int_\Gamma \E L^y_{H_{\ell_-}\land
  H_{\ell_+}} \, \mu(\dd y)\\
    &\le& \frac{2\epsilon \ell_+|\ell_-|}{\ell_+-\ell_-}\mu(\Gamma).
  \end{eqnarray}
 In consequence, for $\epsilon< (2\ell_+|\ell_-|/(\ell_+-\ell_-))^{-1}$, $\spm_\epsilon\in
  S^\mu$. We want to take the maximal element of $S^\mu$ and the
  following lemma describes the key property for our proof.
  \begin{lemma}
    \label{lem:max}
    Suppose that $\spm_1,\spm_2\in S^\mu$. Then the 
measure\footnote{See the proof of Lemma~\ref{lem:max} for a detailed 
definition of $\spm$.} $\spm= \max\{\spm_1,\spm_2\}$ is also
    an element of $S^\mu$.
  \end{lemma}
  The proof of the Lemma, perhaps the most interesting element of the
  proof of Theorem \ref{thm:main}, is postponed.  Using the Lemma we
  can conclude that there exists a maximal element $\spm^{max}\in
  S^\mu$. We claim that $B_{T^{\spm^{max}}}\sim \mu$. Suppose the
  contrary and let $\nu\sim B_{T^{\spm^{max}}}$. As $\spm^{\max}$ is
  an element of $S^\mu$, $\nu$ is dominated by $\mu$ on $(\ell_-,\ell_+)$ and
  with our assumption there exists $\Gamma \subset (\ell_-,\ell_+)$ such 
that
  $\nu(\Gamma)<\mu(\Gamma)$. Let $f$ be the Radon-Nikodym derivative
  of $\nu$ with respect to $\mu$ on $(\ell_-,\ell_+)$. Then there exists an
  $\epsilon>0$ and $\Gamma'\subseteq \Gamma$ such that $f<1-\epsilon$
  on $\Gamma'$ and $\mu(\Gamma')>0$. Let
  $\spm'=\spm^{max}+\gamma\mu\mathbf{1}_{\Gamma'}$ with
  $\gamma=\epsilon (\ell_+-\ell_-)/(4\ell_+|\ell_-|)$ and let $\rho\sim
  B_{T^{\spm'}}$. The measure $m'$ involves extra stopping in
  $\Gamma'$, when compared with $m$, so that necessarily there is less
  chance in stopping off $\Gamma'$. Hence, $\rho\le \nu\le \mu$ on
  $(\ell_-,\ell_+)\setminus \Gamma'$. Moreover, using arguments as above we see
  that $\rho\le \nu+\epsilon/2\le \mu$ on $\Gamma'$ and thus $\spm'\in
  S^\mu$ which contradicts maximality of $\spm^{max}$. Finally,
  $T^{\spm^{max}}$ is minimal since it is
  smaller than $H_{\ell_-}\land H_{\ell_+}$.\smallskip\\
  \emph{Part I: Step 2}. Consider $\mu$ any centered probability
  measure and write $\ell_-,\ell_+$ respectively for the lower and the upper
  bound of the support of $\mu$. We have just treated the case when
  both $\ell_-,\ell_+$ are finite so we suppose that at least one of them is
  infinite. Let $a(n),b(n)$ be two sequences with $a(n)\searrow \ell_-$,
  $b(n)\nearrow \ell_+$, as $n\to \infty$, such that the measure
  \[
  \mu_n=\mu((-\infty,a(n)])\delta_{a(n)}+\mu_{|(a(n),b(n))}+\mu([b(n),\infty))
  \delta_{b(n)}
  \]
  is centered. This measure can be embedded using $T^{\spm_n}$ where
  $\spm_n=\spm^{max}$ is the maximal element of $S^{\mu_n}$. Clearly
  $\spm_k\in S^{\mu_n}$ for all $k\ge n$ and thus $(\spm_n)$ is a
  decreasing sequence. It thus converges to a limit denoted
  $\spm=\inf_{n}\spm_n$, which is a measure (see e.g.\ Sec.~III.10 in
  \cite{Doob:94}).\\
  Let $\Delta$ and $\td{\Delta}_n$, $n\ge 1$, be independent Poisson
  measures with intensities respectively $\spm$ and
  $(\spm_n-\spm_{n+1})$, $n\ge 1$. Consider
  $\Delta_n=\Delta+\sum_{k\ge n}\td{\Delta}_k$ which is again a
  Poisson point measure with intensity $\spm+\sum_{k\ge
    n}(\spm_k-\spm_{k+1})=\spm_{n}$. With $T^{\spm_n}=\inf\{t:
  \Delta_n(R_T)\ge 1\}$ as previously, we have that
  $B_{T^{\spm_n}}\sim \mu_n$. Furthermore, as $\Delta_n(\Gamma)\ge
  \Delta_{n+1}(\Gamma)\ge \ldots$ we have that $T^{\spm_n}\le
  T^{\spm_{n+1}}\le \ldots$ so that $T^{\spm_n}\nearrow T=T^\spm$ as
  $n\to \infty$. To show that $T^\spm<\infty$ a.s. we recall that
  $|B_t|-L_t^0$ is a martingale. As $T^{\spm_n}\le H_{a(n),b(n)}$ we
  have that
  \[
  \E L^0_{T^{\spm_n}}=\E |B_{T^{\spm_n}}|\le \int_{-\infty}^\infty
  |x|\mu(\dd x).
  \]
  The left hand side converges to $\E L^0_{T^\spm}$ which is thus
  finite and in particular $T^\spm<\infty$ a.s. 
  Finally, since $\spm\in S^{\mu_n}$ for all $n$, the law of $B_{T^\spm}$ is dominated by
  $\mu$ on $\R$ and is thus simply equal to $\mu$.  The uniform
  integrability of $(B_{t\land T^\spm}:t\ge 0)$ follows from standard
  arguments (\eg{} Proposition~18 of \cite{Cox:08}).
  \smallskip\\
  \emph{Part II}. To show that $\td{T}^\spm$ is exponentially
  distributed, we recall the above definitions. Then for $t>s$:
  \begin{eqnarray*}
    \Pr\left(\td{T}^\spm \ge t | \td{T}^\spm \ge s\right) & = &
    \Pr\left(A^{-1}_{T^\spm} \ge t | A^{-1}_{T^\spm} \ge s \right)\\
    & = & \Pr\left(T^\spm \ge A_t | T^\spm \ge A_s\right)\\
    & = & \Pr\left(\Delta(R_{A_t-}) = 0 |  \Delta(R_{A_s-}) = 0\right)\\
    & = & \Pr\left(\Delta(R_{A_t-} \setminus R_{A_s-}) = 0\right) \\
    & = & \Pr\left(\Delta(R_{A_t} \setminus R_{A_s}) = 0\right)
  \end{eqnarray*}
  where for the last equality we use the fact that local times are
  continuous in $t$.  However, conditional on $B_t$, we know
  $\Delta(R_{A_t} \setminus R_{A_s})$ is Poisson with parameter
  \begin{equation*}
    \int_{-\infty}^\infty (L_{A_t}^x - L_{A_s}^x) \, \spm(\dd x)=\Phi_{A_t}-\Phi_{A_s}=t-s
  \end{equation*}
  where we used \eqref{eq:phi_inv_A}. Clearly $X_{\td{T}^m}=B_{T^m}$.
  Finally a similar calculation to the ones above shows that
  $\Pr(\td{T}^m>t|\sigma(X_u:u\le t))=\Pr(\td{T}^m>t)$ and $\td{T}^m$
  is independent of $X$. This completes the proof of Theorem
  \ref{thm:main2}.
\end{proof}

\begin{proof}[Proof of Lemma \ref{lem:max}]
  Note that the measure $\spm$ is well defined. More precisely let
  $\spm_3=\spm_1+\spm_2$ and $f_1$ and $f_2$ the Radon-Nikodym
  derivatives respectively of $\spm_1$ and $\spm_2$ with respect to
  $\spm_3$. The measure $\spm$ is defined via its Radon-Nikodym
  derivative $f=f_1\lor f_2$ with respect to $\spm_3$. Likewise, the
  measure $\underline{\spm}=\min\{\spm_1,\spm_2\}$ is well defined. We
  write $\nu^{\spm}$ for the law of $B_{T^{\spm} \wedge H_a\wedge
    H_b}$. Note that by construction
  $\nu^{\underline{\spm}}(\Gamma)\leq \nu^{\spm_1}(\Gamma)\leq
  \mu(\Gamma)$ and hence $\underline{\spm}\in S^\mu$.\\
  Consider the signed measure $(\spm_1-\spm_2)$ and let $F_1$ be the support
  of its positive part and $F_2$ the support of its negative part.
  Then we may decompose $\Delta$ associated with $\spm$ into three
  independent Poisson Random Measures, $\Delta_\wedge$ with intensity
  $\dd u\,\underline{\spm}(\dd x)$, $\Delta_1$ with intensity $\dd
  u\,(\spm_1(\dd x)-\spm_2(\dd x))\indic{F_1}(x)$ and $\Delta_2$ with
  intensity $\dd u\, (\spm_2(\dd x)-\spm_1(\dd x)) \indic{F_2}(x)$. We
  know that the stopping times generated by the measures
  $\Delta_\wedge + \Delta_1$ and $\Delta_\wedge + \Delta_2$ both lead
  to measures which are dominated by $\mu$ on $(a,b)$. We wish to
  deduce the same about $\Delta_\wedge + \Delta_1+ \Delta_2$.
  \begin{figure}[t]
    \begin{center}
      \includegraphics[width=3in,height=3in]{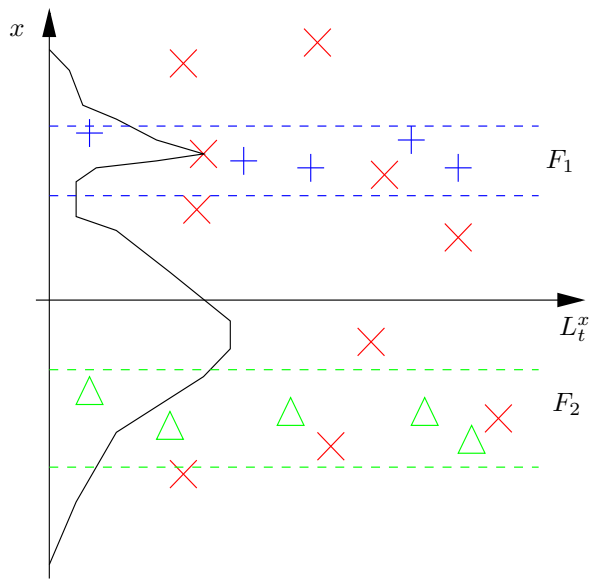}
      \caption{\label{fig:max} Representation of the Poisson Random
        Measure $\Delta$ in terms of $\Delta_{\wedge}, \Delta_1$ and
        $\Delta_2$, represented by $\times, +$ and $\triangle$
        respectively}
    \end{center}
  \end{figure}
  \\
  We show this by considering the coupling implied by
  Figure~\ref{fig:max}. Given a set $\Gamma\subset (a,b)$, we need to
  show that $\nu^\spm(\Gamma) \le \mu(\Gamma)$. However by considering
  $\Gamma \subseteq F_1$, it is clear that $\nu^\spm (\Gamma) \le
  \nu^{\spm_1}(\Gamma)$ since adding adding extra marks according to
  $\Delta_2$ can only reduce the probability of stopping in $F_1$, as it will not
  produce any new `points' in this set. Similarly, for $\Gamma
  \subseteq F_2$, we will have $\nu^\spm (\Gamma) \le \nu^{\spm_2}
  (\Gamma)$. Finally, for $\Gamma \subseteq (a,b) \setminus(F_1 \cup
  F_2)$ we have $\nu^\spm (\Gamma) \le \nu^{\underline{\spm}}
  (\Gamma)$. It now follows from $\spm_1,\spm_2,\underline{\spm}\in
  S^\mu$ that $\spm \in S^\mu$.
\end{proof}
We have thus proved existence of a suitable measure $m$ such that the
generalised diffusion $(X_t)$ associated to $m$ satisfies $X_T\sim
\mu$ for an independent exponential time $T$. We have also seen that
this is equivalent to finding $m$ such that $B_{T^m}\sim \mu$, where
$T^m$ is stopping time defined in \eqref{eqn:Tdefn}. We can use this
new interpretation to recover the formula \eqref{eq:m_measure_def} for
$m$. Indeed, from construction of $T^m$, intuitively we have
$$\Pr(B_{T^m}\in \dd x) \ = \ m(\dd x) \ \times \ \E\left(\textrm{time 
    spent in }\{x\}\textrm{ by }(B_t:t\leq T^m)\right).$$ The time
spent in $\{x\}$ by $(B_t:t\leq T^m)$ is simply $L^x_{T^m}$ and
$\E[L^x_{T^m}]=\E|B_{T^m}-x|-|x|$. Hence, if we are to have
$B_{T^m}\sim \mu$ we have to have
\[
\mu(\dd x)\ = \ m(\dd x)\ \times \ \left(\int |x-y|\mu(\dd
  y)-|x|\right)\ = \ m(\dd x)\ \times \ \left(u_\mu(x)-|x|\right),
\]
which is exactly \eqref{eq:m_measure_def}.

\bibliography{timehomogconstr} \bibliographystyle{elsart-harv}

\end{document}